\documentclass[a4paper]{amsart}
\usepackage{amsfonts}\usepackage{amssymb}\usepackage{amsmath}\usepackage{amsthm}
\usepackage{latexsym}
\usepackage{graphicx}
\usepackage{layout}
\usepackage[T1]{fontenc}
\usepackage[utf8]{inputenc}
\usepackage[colorlinks=true,urlcolor=blue,citecolor=red,linkcolor=blue,linktocpage,pdfpagelabels,bookmarksnumbered,bookmarksopen]{hyperref}

\newtheorem{theorem}{Theorem}[section]

\newtheorem{lemma}[theorem]{Lemma}

\newcommand{\sezione}[1]{\section{#1}\setcounter{equation}{0}}

\def\R{\mathbb{R}}
\def\C{\mathbb{C}}

\def\e{{\varepsilon}}
\def\di12{\mathcal{D}^{1,2}(\R^n)}

\def\l{{\lambda}}

\def\0l{_{0,\l}}
\def\1l{_{1,\l}}
\def\2l{_{2,\l}}
\def\3l{_{3,\l}}
\def\4l{_{4,\l}}

\def\Om{\Omega}

\def\sideremark#1{\ifvmode\leavevmode\fi\vadjust{\vbox to0pt{\vss
 \hbox to 0pt{\hskip\hsize\hskip1em
 \vbox{\hsize2.1cm\tiny\raggedright\pretolerance10000
  \noindent #1\hfill}\hss}\vbox to15pt{\vfil}\vss}}}%

\newtheorem*{theorem*}{Theorem}

\begin{document}
\title[critical points]{On the number of critical points of solutions of semilinear equations in $\R^2$}
\thanks{This work was supported by Prin-2015KB9WPT, by Universit\'a di Roma "La Sapienza" and partially supported by Indam-Gnampa}
\author[Gladiali]{Francesca Gladiali}
\address{Dipartimento di Chimica e Farmacia, Universit\`a di Sassari, via Piandanna 4 - 07100 Sassari, e-mail: {\sf fgladiali@uniss.it}.}
\author[Grossi]{Massimo Grossi }
\address{Dipartimento di Matematica, Universit\`a di Roma ``La Sapienza", P.le A. Moro 2 - 00185 Roma, e-mail: {\sf massimo.grossi@uniroma1.it}.}

\maketitle
\begin{abstract}
In this paper we construct families of bounded domains $\Om_\e$ and solutions  $u_\e$ of
\[\begin{cases}
-\Delta u_\e=1&\text{ in }\ \Om_\e\\
u_\e=0&\text{ on }\ \partial\Om_\e
\end{cases}\]
such that, for any integer $k\ge2$, $u_\e$ admits at least $k$ maximum points for small enough $\epsilon$. The domain  $\Om_\e$ is ``not far'' to be convex in the sense that it
 is starshaped, the curvature of  $\partial\Om_\e$ vanishes at exactly $two$ points and the minimum of  the curvature of  $\partial\Om_\e$ goes to $0$ as $\e\to0$. 
\end{abstract}

\sezione{Introduction}\label{s0}
The computation of the number and of the nature of critical points of positive solution of the problem
\begin{equation}\label{i0}
\begin{cases}
-\Delta u=f(u)&\text{ in }\ \Om\\
u=0&\text{ on }\ \partial\Om
\end{cases}
\end{equation}
where $\Om\subset\R^n$, $n\ge2$ is a smooth bounded domain and $f$ is a smooth nonlinearity, is a classic and fascinating problem.\\
Many  techniques and important results were developed  in the literature (Morse theory, degree theory, etc.) to address this problem. In these few lines is impossible to mention all these contributions, so we will limit ourselves to recall some of them that are closer to the purpose of this paper.\\
One of the first major results concerns the case $f(u)=\l u$, so $u$ is the first eigenfunction of the Laplacian with zero Dirichlet boundary condition. It was proved by
 Brascamp and Lieb \cite{bl} and  Acker, Payne and Phillippin \cite{app} in dimension $n=2$ that if $\Om\subset\R^n$ is strictly convex then $-\log u$ is convex, so that the superlevel sets are convex and $u$ admits a unique critical point in $\Omega$. Other results on the shape of level sets for various nonlinearities $f$ can be found in \cite{Ko83}, \cite{CS82}, \cite{k}, \cite{Ka86}, \cite{Ga55}, \cite{Ga57}, \cite{Cslin94} and references therein.
 \\
 A second seminal result that we want to mention is the fundamental theorem by Gidas, Ni and Nirenberg \cite{gnn}, which holds in domains which are convex in the direction $x_i$ for any $i=1,..,n$. We have that a domain is convex in the direction $x_1$ (say) if $P=(p_1,x')\in\Omega$ and $Q=(q_1,x')\in\Omega$ then the line segment $\overline{PQ}$ is contained in $\Omega$.
\\
 
\begin{theorem*}[\bf Gidas, Ni, Nirenberg]
 Let $\Om\subset\R^n$ a bounded, smooth domain which is symmetric with respect to the plane $x_i=0$ for any $i=1,..,n$ and  convex in the $x_i$ direction for $i=1,..,n$. Suppose that $u$ is a positive
 solution to \eqref{i0}
where $f$ is a locally Lipschitz nonlinearity. Then 
\begin{itemize}
\item $u$ is symmetric with respect to $x_1,..,x_n$. (Symmetry)
\item  $\frac {\partial u}{\partial x_i}<0$ for $x_i>0$ and $i=1,\dots,n$. (Monotonicity)
\end{itemize}
\end{theorem*}
\noindent An easy consequence of the symmetry and monotonicity properties in the previous theorem is that \[
 \sum_{i=1}^nx_i\frac{\partial u}{\partial x_i}<0 \ \ \forall x\ne0\] that is all the superlevel sets are $starshaped$ with respect to the origin.\\
This theorem holds in symmetric domains. Although it is expected that the uniqueness of the critical point (as well as the starlikeness of superlevel sets) holds in more general convex domains, this is a very difficult hypothesis to remove. \\
Next we mention another important result which holds for a wide class of nonlinearities $f$ without the symmetry assumption on $\Om$ and for semi-stable solutions. To this end we recall that a solution $u$ to \eqref{i0} is semi-stable if the linearized operator at $u$ admits a nonnegative first eigenvalue.
\begin{theorem*}[\bf Cabr\'e, Chanillo \cite{cc}]
 Assume $\Omega$ is a smooth, bounded and convex domain of $\R^2$ whose boundary has positive curvature. Suppose $f\ge0$ and $u$ is a semi-stable positive solution to \eqref{i0}. 
Then $u$ has a unique critical point, which is non-degenerate.
\end{theorem*}

As a consequence the superlevel sets of $u$ are strictly convex in a neighborhood of the critical point and in a neighborhood of the boundary. It is thought that they are all convex, but this is certainly not true for suitable nonlinearities like  in the following surprising result:

 \begin{theorem*}[\bf Hamel, Nadirashvili, Sire \cite{hns}]
 In dimension $n=2$ there are some smooth bounded convex domains $\Omega$ and some $C^{\infty}$ functions $f:[0,+\infty)\to \R$ for which problem \eqref{i0} admits a solution $u$ which is not quasiconcave.
 \end{theorem*}
 We recall that a function is called quasiconcave if its superlevel sets are all convex. We can then conclude that the convexity of the domain is not  always preserved by the superlevel sets. Nevertheless by the Gidas, Ni, Nirenberg theorem, being the domain $\Omega$ in \cite{hns} symmetric, the superlevel sets in this example are still starshaped  and the maximum point of the solution is unique.\\

\noindent The previous results suggest the following questions:
\vskip0.2cm
\noindent {\bf Question 1}: {\em Assume $\Omega$ is starshaped. Are the superlevel sets of any positive solution to \eqref{i0} starshaped?}
\vskip0.2cm

\noindent {\bf Question 2}: {\em Assume that $u$ is a positive solution to \eqref{i0} 
in a smooth bounded domain $\Om\subset\R^2$ whose curvature is negative somewhere. What about the number of critical points of $u$?}
\vskip0.2cm
\noindent Of course interesting examples deal with contractible domains $\Omega$, otherwise it is not difficult to construct examples of solution $u$ to \eqref{i0} with many critical points.
Some results in the direction to prove Question 1 were obtained for non-symmetric domains, in a perturbative setting, by Grossi and Molle \cite{gm} and Gladiali and Grossi \cite{gg1,gg2}.\\
In this paper we answer Question 1 showing that the starlikeness of the domain is not maintained by the superlevel sets.
Moreover we consider also Question 2 showing that in general there is no bound on the number of critical points. \\
Of course this last result is very sensitive to the shape of $\Om$. In a recent paper \cite{lr} it was showed that if $\partial\Om$ is contained in $\left\{z\in\C:|z|^2=f(z)+\overline{f(z)}\right\}$ where $f(z)$ is a rational function, then, differently than our case, there is a  bound on the number of the critical points. We refer to \cite{am} for other results in these direction.
\\
Actually we will construct  a family of domains $\Om_\e$
starshaped with respect to an interior point and solutions $u_\e$ of the classical torsion problem, namely 
\begin{equation}\label{eq:torsion}
\begin{cases}
-\Delta u=1&\text{ in } \ \Om\\
u=0&\text{ on }\ \partial\Om
\end{cases}
\end{equation}
with an arbitrary large number of  maxima and of disjoint superlevel sets. 
Moreover the curvature of $\partial\Om_\e$ vanishes at exactly two points and its minimum value goes to $0$ as $\e\to0$.
In some sense our domains $\Om_\e$  are not ``far'' to be convex.
More precisely our result is the following,
\begin{theorem}\label{i1}
For any integer  $k\geq 2$ there exists a  family of smooth bounded domains $\Om_{\e,k}\subset\R^2$ and smooth functions $u_{\e,k}:\Om_{\e,k}\to\R^+$ which solves the torsion problem \eqref{eq:torsion} in $\Omega_{\e,k}$, such that for $\e$ small enough,
\begin{itemize}
\item $\Om_{\e,k}$ is  starshaped with respect to an interior point. \hfill $(P0)$
\item The set $u_{\e,k,}$ $\{u_{\e,k}>c\}$ is non-empty and has at least
$k$ connected components; in particular $u_{\e,k}$ has at least $k$ maximum points.\hfill $(P1)$ 
\item If $S$ is the strip $S=\{(x,y)\in\R^2\hbox{ such that }|y|<1\}$ 
and $Q$ is any compact set of $\R^2$ then \  $\Om_{\e,k}\cap Q\xrightarrow[\e\to0]\ S\cap Q$.\hskip5cm(P2)
\item The curvature of $\partial\Om_{\e,k}$ changes sign and vanishes exactly at two points. Moreover  $\min\Big(Curv_{\partial\Om_{\e,k}}\Big)\xrightarrow[\e\to0]\ 0$
.\hfill $(P3)$
\end{itemize}
\end{theorem}
A picture of $\Om_{\e,2}$ for $\e$ small is given in Fig.1.
\begin{figure}[h]
\centering
\includegraphics[scale=0.15]{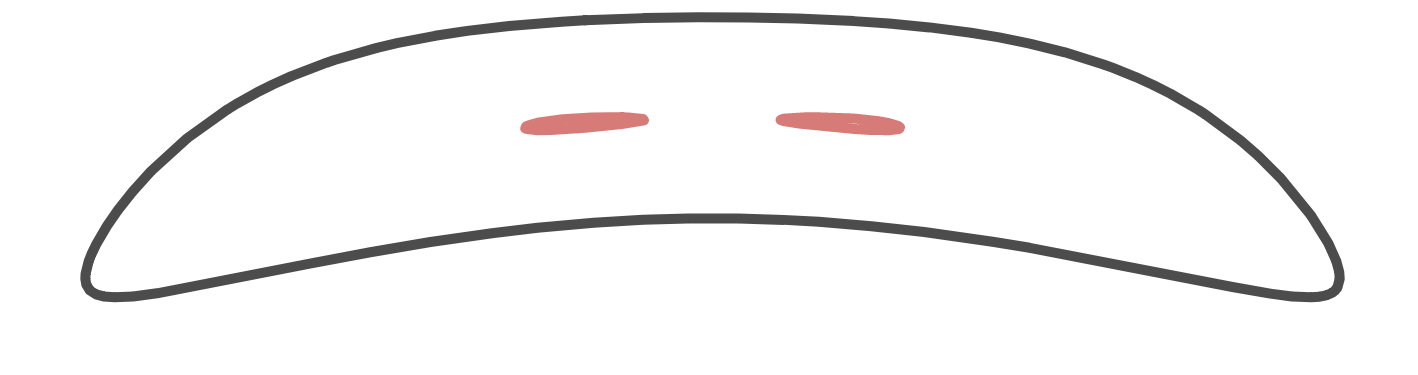}
\caption{Domain $\Om_{\e,2}$ with  level set $\{u_{\e,2}=c\}$}
\end{figure}
Of course $(P2)$ implies that the superlevel set $\{u_{\e,k}>c\}$ is not starshaped.
We recall that every solution to \eqref{eq:torsion} is positive by the Maximum principle and semi-stable as in \cite{cc}.
We point out that the solution $u_{\e,k}$ will be explicitly provided and the domain $\Om_{\e,k}$ will be the superlevel set $\{u_{\e,k}>0\}$.\\
In some sense our result shows that the assumption on the positivity of the curvature of $\partial\Om$ in
 Cabr\'e and Chanillo's Theorem cannot be relaxed because it is enough that the curvature of $\partial\Om_{\e,k}$ satisfies $(P3)$ to imply that there exists a semi-stable solution of a (simple) PDE with an arbitrary number of critical points. By $(P2)$ our domain is 'locally''' close to a strip and $u_{\e,k}\xrightarrow[\e\to0]\ \frac12-\frac{y^2}2$ in $C^2_{loc}(\R^2)$. Note that the function $\frac12-\frac{y^2}2$, which solves $-\Delta u=1$ in the strip $S$, was also used in Hamel, Nadirashvili and Sire \cite{hns}.
We point out that when $\e$ is small enough, the domain $\Omega_{\e,k}$ in Theorem \ref{i0} looks like  the one in  \cite{hns} even if it has negative curvature somewhere.

Before describing the construction of the solution $u_{\e,k}$ let us make some remarks on  $(P2)$. It proves that the starlikeness of $\Om_{\e,k}$ is not enough to guarantee that the superlevel sets are starshaped proving Question 1. To our knowledge this is the first example with this property. 
Theorem \eqref{i1} also shows that it cannot exist a starshaped rearrangement which associates to a smooth function $u$ another function $u^*$ with starshaped superlevel sets verifying  the standard properties of rearrangements, i.e.
\begin{equation}\label{i2}
\int_{\Om^*}|u^*|^p= \int_\Om|u|^p\quad\forall p\ge1\qquad\hbox{ and }\qquad\int_{\Om^*}|\nabla u^*|^2\le \int_\Om|\nabla u|^2.
\end{equation}
A starshaped rearrangement which verifies, under additional assumptions, properties \eqref{i2} was introduced by Kawohl in \cite{k1} and \cite{k}. This implies that, jointly with \eqref{i2},
\begin{equation}\label{i99}
\inf\limits_{u\in H^1_0(\Om)}\frac12\int_\Om|\nabla u|^2-\int_\Om u
\end{equation}
is achieved at a unique function $u$ with starshaped superlevel sets. However this type of rearrangement does not always exist, it is depending on the shape of $\Om$.
An example (see \cite{k} and \cite{g}) is the so-called {\em Grabmüller's long nose''} \cite{g}. Since our pair $(\Om_{\e,k},u_{\e,k})$ satisfies \eqref{i99} and $u_{\e,k}$ has some superlevel sets which are not starshaped, the requested starshaped rearrangement cannot exists for $\Om_{\e,k}$.\\
Finally we remark that in Makar-Limanov \cite{ml} it was proved that if $\Omega$ is a smooth bounded strictly convex domain of $\R^2$ and $u$ solves the torsion problem in $\Omega$ then the superlevel sets are strictly convex too. It seems then that the torsion problem is a {\em``good'' } problem in which the properties of $\Omega$ are maintained by the superlevel sets. It is then even more unexpected that this does not hold for the starlikeness.

Next we say some words about the construction of $u_{\e,k}$. The starting point is given by the function
$$\phi(y)=\frac 12-\frac12y^2$$
which solves
\begin{equation}
\begin{cases}
-\Delta\phi=1&\text{ in }\ |y|<1\\
u=0&\text{ on }\ y=\pm1.
\end{cases}
\end{equation}
Our function $u_{\e,k}$ is a perturbation of $\phi$ with suitable $harmonic$ functions. The choice of the harmonic functions is quite delicate: let us consider the holomorphic function $F_k:\mathbb{C}\to\mathbb{C}$,
\begin{equation}\label{eq:F}
F_k(z)=-\Pi_{i=1}^{2k}(z-x_i)
\end{equation}
for arbitrary real numbers $x_1<x_2<..<x_{2k}$ and define
$$v_k(x,y)=Re\Big(F_k(z)\Big).$$
Next we define  $u_{\e,k}$  as
$$u_{\e,k}(x,y)=\frac 12-\frac12y^2+\e(y^3-3yx^2)+\e^\frac32v_k(x,y).$$
We have trivially that $-\Delta u_{\e,k}=1$ and the proof of Theorem \ref{i1} reduces to show that for $\e$
small enough the set $\Om_{\e,k}=\{u_{\e,k}>0\}$ is a bounded smooth domain which verifies $(P0)-(P3)$. Although the function $u_{\e,k}$ is explicitly  provided, the proof of  Theorem \ref{i1} involves delicate computations. Note that the power $\frac32$ appearing in the definition of $u_{\e,k}$ can be replaced with any real number $\alpha\in(1,2)$. However $\alpha=2$ is not allowed for technical reasons (``bad'' interactions occur).\\
There is a flexibility in the choice of the holomorphic function $F_k$; indeed it can be replaced by another one such that the restriction to the real line has $k$ maxima points and verifies some suitable growth condition at $\pm\infty$.  \\

Theorem \ref{i0} can be extended to semi-stable solutions of more general nonlinear problems. Let us  consider a solution $u$  to
\begin{equation}\label{f1}
\begin{cases}
-\Delta u=\lambda f(u)&\hbox{in }\Om\\
u>0&\hbox{in }\Om\\
u=0&\hbox{on }\partial\Om
\end{cases}
\end{equation}
where $\Om\subset\R^2$ is a bounded smooth domain, $f:\R^+\to\R$ is a smooth nonlinearity (say $C^1$) with $f(0)>0$ and $u_\l$ is a family of solutions of \eqref{f1} satisfying
\begin{equation}\label{f2}
||u_\l||_\infty\le C\quad\hbox{for $\l$ small},
\end{equation}
with $C$ independent of $\l$. 
A classical example of solutions satisfying \eqref{f1} and \eqref{f2} was given by Mignot and Puel \cite{mp} when $f$ is a positive, increasing and convex nonlinearity and $0<\l<\l^*$. 
See \cite[Theorem 10]{gg1} for some results about convexity and uniqueness of the critical point to solutions to \eqref{f1}.
Then we have the following result,
\begin{theorem}\label{i3}
Let $\e>0$, $k\ge 2$ and $\Omega_{\e,k}$ be as in Theorem \ref{i0}. Then there exists $\bar \l$ (depending on $\e$) such that if $u_{\l,\e,k}$
is a solution to \eqref{f1} in $\Omega_{\e,k}$ that satisfies \eqref{f2}, we have that, for any $0<\l<\bar \l$, $u_{\l,\e,k}$
 is semi-stable and satisfies  $(P1)$-$(P3)$.
\end{theorem}

\section{The holomorphic function $F(z)$}
Here and in the next sections, to simplify the notation we omit the index $k$ when we define the functions $v(x,y)$, $u_\e(x,y)$ and the domains $\Omega_\e$.\\
For $k\ge2$ let us consider arbitrary real numbers
$x_1<x_2<..<x_{2k}$ and the holomorphic function $F:\C\to\C$ in \eqref{eq:F} given by
\begin{equation}\label{b1}
F(z)=-\Pi_{i=1}^{2k}(z-x_i)=-\sum_{i=0}^{2k}a_iz^i
\end{equation}
where of course $a_{2k}=1$.\\
Let us denote by $f$ the restriction of $F$ to the {\em real line}. We immediately get that $f(x_1)=..=f(x_{2k})=0$ and that $f$ has $k$ maximum points. 
Let us consider the function $v:\R^2\to\R$ defined as
\begin{equation}\label{eq:v-re}
v(x,y)=Re\Big(F(z)\Big)
\end{equation}
which is harmonic  in $\R^2$ and satisfies $v(x,0)=f(x)$. By construction
\begin{equation}\label{eq:v-pj}
v(x,y)=-\sum _{j=0}^{2k} a_jP_j(x,y)\end{equation}
with $a_{2k}=1$ and where $P_j$ are homogeneous harmonic polynomials of degree $j$.  Finally we introduce the function \begin{equation}\label{eq:u-epsilon}
\boxed{u_\e(x,y)=\frac 12-\frac12y^2+\e(y^3-3x^2y)+\e^\frac 32v(x,y)}
\end{equation}
which satisfies
\[-\Delta u_\e=1\ \ \text{ in }\R^2.\]
The function $v$ coincides with $f(x)$ along the $x$-axis, while $u_\e(x,0)=\frac 12-f(x)$. We end this section with a brief comment on the term $\e(y^3-3x^2y)$: it appears in the definition of $u_\e$ to have that the curvature of our domain vanishes exactly at two points. It
breaks the symmetry of the domain with respect to $y$, otherwise we would have that the curvature vanishes at $four$ points.

\section{Proof of Theorem \ref{i0}}
In this section we show that the function $u_\e$ in \eqref{eq:u-epsilon} verifies the claim of Theorem \ref{i0}. In the rest of the paper we let $o(1)$ be  a quantity that goes to zero as $\e$ goes to zero and $k\geq 2$.
\begin{theorem}\label{b2}
For $\e$  small enough the function $u_\e(x,y)$ in \eqref{eq:u-epsilon} 
admits a connected component (that we call $\Omega_\e$) of the superlevel set
\[
\{(x,y)\in\R^2\hbox{ such that }u_\e(x,y)>0\}
\]
which satisfies:\\
i) $\Omega_{\e}$ is a smooth bounded domain;\\
ii) $\Omega_{\e}$ is starshaped with respect one of its points;\\
iii) $\Omega_{\e}$ contains $k$ disjoint connected components $Z_{1,\e},..,Z_{k,\e}$ of the superlevel set 
$\{(x,y)\in\R^2\hbox{ such that } u_\e(x,y)>\frac 12\}$.
\end{theorem}
\begin{proof}
{\bf Step 1:} Let $x_\e=\left(\frac 3{\e^\frac 32}\right)^\frac 1{2k}$. We want to show that 
\[u(\pm x_\e,y)\leq -2 \ \ \text{ for } |y|<1+h\]
when $\e$ is small enough and $0<h<1$. In \eqref{eq:v-pj} 
let us consider the polynomial of degree $2k$, namely
\[P_{2k}(x,y)=\sum_{j=0}^k b_j x^{2k-2j}y^{2j}\]
for some suitable coefficients $b_j$ such that $b_0=b_{k}=1$. Then
\[\e^\frac 32 P_{2k}(x_\e,y)=\e^\frac 32 \sum_{j=0}^k b_j   \left(\frac 3{\e^\frac 32}\right)^\frac {2k-2j}{2k} y^{2j}
=3+o(1)\quad\hbox{as }\e\to0\]
uniformly with respect to $-1-h<y<1+h$.\\
In a very similar manner, for any $0\le j\le2k-1$ we have that
\[\e^\frac 32 P_j(x_\e,y)=o(1)\]
and
 \[\left| \e (y^3-x_\e^2y)\right|=O\left(\e^\frac {2k-3}{2k}\right)\]
for $\e\to 0$ uniformly with respect to $-1-h<y<1+h$. 
Considering all these estimates we obtain
 \[\left| u_\e(x_\e,y)+\frac 52+\frac 12 y^2\right|=o(1)\quad\hbox{as }\e\to0.
\]
The very same computation shows also that 
\[ u_\e(-x_\e,y)\leq -2
\]
when $\e$ is small enough and concludes the proof.

\

\noindent {\bf Step 2:} We show that $u_\e(x,y)< 0$ on the segments 
$$T_{\pm h}=\{(x,y)\in\R^2\hbox{ such that } y=\pm(1+ h), x\in[-x_\e,x_\e]\}$$
 for some $0<h<1$ when $\e$ is small enough.\\
First let us observe that for 
$(x,y)\in T_{\pm h}$,
\[\left|\e(y^3-3x^2y)\right|\leq \e (8+6x_\e^2)\leq 8\e+12\e^\frac{2k-3}{2k}=O\left(\e^\frac{2k-3}{2k}\right)\]
when $\e\to 0$.
Next, note that, by \eqref{eq:v-pj}
\[v(x,\pm(1+h))=-\sum_{j=1}^{2k}a_jP_j(x,\pm( 1+h))\]
and since  $a_{2k}=1$ we get that
\[\sup_{x\in \R}v(x,\pm(1+h))=C\in\R.\]
Then we  obtain
\[u_\e(x,\pm(1+h))=-\frac 12h^2-h+\e ((\pm(1+h))^3\pm3x^2(1+h))+\e ^\frac 32 v(x,\pm(1+h))<-\frac 12 h^2<0\]
for $\e $ small enough.\\

\noindent {\bf Step 3:}
 We have proved that 
 for every $\e$ small enough $u_\e(x,y)<0$
  on the boundary of the rectangle $R_\e=\{(x,y)\in\R^2\hbox{ such that } -x_\e\le x\le x_\e, -(1+h)\le y\le1+h \}$. Since $u_\e(x_1,0)=\frac 12+\e v(x_1,0)=\frac 12+\e f(x_1)=\frac 12$
   this implies that there is a connected component of the superlevel set $u_\e(x,y)>0$
   , that we call $\Omega_{\e}$, which is contained in the interior of $R_\e$ and contains the point $(x_1,0)$.
Since $u_\e$ is continuous then $\Omega_{\e}$ is a connected
open set with non empty interior.\\
Furthermore when $\e$ satisfies
\begin{equation}\label{cond-ep-1}
\e<\left(\frac 1{2\sup_{x\in[x_1,x_{2k}] }(-f(x))}\right)^\frac 23
\end{equation}
then all the segment $[x_1,x_{2k}]\times \{0\}$ belongs to $\Omega_\e$.\\

\noindent {\bf Step 4:}
In this step we prove that when $\e$ is small enough $\Omega_{\e}$ is $smooth$ and $starshaped$ with respect to the point $(x_1,0)$, which is equivalent to show that 
$$(x-x_1,y)\cdot\nu(x,y)\le-\alpha<0\hbox{ for any }(x,y)\in\partial \Omega_{\e},$$
where $\nu(x,y)$ is the outer normal of $\partial \Omega_\e$ at the point $(x,y)$. In particular we will show that
\begin{equation}\nonumber
(x-x_1)\frac{\partial u_\e}{\partial x}+y\frac{\partial u_\e}{\partial y}\le-\alpha\quad\forall (x,y)\in R_\e\hbox{ such that } u_\e(x,y)=0.
\end{equation}
It is easily seen that
\[(x-x_1)\frac{\partial u_\e}{\partial x}+y\frac{\partial u_\e}{\partial y}=-y^2+\e\left( -9x^2y+6xx_1y+3y^3\right) +\e^\frac 32 \left((x-x_1)\frac{\partial v}{\partial x}+y\frac{\partial v}{\partial y}\right).\]
On the other hand since $u_\e(x,y)=0$ on $\partial \Omega_{\e}$ we get that
\[-y^2=-1-2\e(y^3-3x^2y) -2\e^\frac 32v(x,y)\]
and 
\begin{equation}\label{eq:stell}\nonumber
(x-x_1)\frac{\partial u_\e}{\partial x}+y\frac{\partial u_\e}{\partial y}=-1+\e\left(y^3-3x^2y+6xx_1y \right) 
+\e^\frac 32 \left((x-x_1)\frac{\partial v}{\partial x}+y\frac {\partial v}{\partial y}-2v(x,y)\right).
\end{equation}
By \eqref{eq:v-pj} and  Euler Theorem we get
\begin{equation}\nonumber
x\frac{\partial v}{\partial x}+y\frac{\partial v}{\partial y}=-\sum_{j=0}^{2k}a_j\left( 
x\frac{\partial P_j}{\partial x}+
y\frac{\partial P_j}{\partial y}\right)=- \sum_{j=1}^{2k} j a_j P_j(x,y) 
\end{equation}
and so, recalling that $a_{2k}=1$,
\[
-2v(x,y)+(x-x_1)\frac{\partial v}{\partial x}+y\frac {\partial v}{\partial y}=- \sum_{j=0}^{2k} (j-2) a_j P_j(x,y)+x_1\sum_{j=1}^{2k} a_j\frac{\partial P_j}{\partial x}\xrightarrow[|x|\to\infty]\ -\infty
\]
uniformly for $y\in[-1-h,1+h]$. Hence
\[\sup_{(x,y)\in (-\infty,\infty)\times [-1-h,1+h]}\left(-2v(x,y)+(x-x_1)\frac{\partial v}{\partial x}+y\frac {\partial v}{\partial y}\right)=d<\infty.\]
In addition
\[\sup_{(x,y)\in [-x_\e,x_\e]\times [-1-h,1+h]} \left(y^3-3x^2y+6xx_1y\right)\le Cx_\e^2=O\left(\e^{-\frac 3{2k}}\right)\]
as $\e\to 0$, so that 
\[\e \left(y^3-3x^2y+6xx_1y\right)2=O\left(\e^{\frac {2k-3}{2k}}\right)\]
in the rectangle $R_\e$.
Summarizing again we have that
\[\sup_{\partial \Omega_\e\subset R_\e}\left(    (x-x_1)\frac{\partial u_\e}{\partial x}+y\frac{\partial u_\e}{\partial y}\right)\le -1+o(1)<-\frac12\]
for $\e\to 0$ which gives the claim.\\
Of course $(x-x_1)\frac{\partial u}{\partial x}+y\frac{\partial u}{\partial y}\neq 0$ on $\partial \Omega_{\e}$ implies that $\partial \Omega_{\e}$ is a smooth curve. \\

\noindent {\bf Step 5:} Here we prove that the superlevel set
 \[L_{\e}:=
\{(x,y)\in \R^2 \hbox{ such that }
u_\e(x,y)> \frac 12
\}
\]
admits in $\Omega_\e$ at least $k$ disjoint components $Z_{1,\e},\dots, Z_{k,\e}$.\\ 
Since $f(x_j)=0$ for $j=1,\dots,2k$ and $f(x)\to -\infty$ as $|x|\to \infty$, there exist points $s_j\in (x_{2j},x_{2j+1})$ for $j=1,\dots,k-1$ and points $\bar s_j\in (x_{2j+1},x_{2j+2})$ for $j=0,\dots,k$ such that
\begin{align*}
f(s_j)=\min_{x\in [ x_{2j},x_{2j+1}]}f(x)<0 & \text{ for }j=1,\dots,k-1\\
f(\bar s_j)=\max_{x\in [ x_{2j+1},x_{2j+2}]}f(x)>0 & \text{ for }j=0,\dots,k-1
\end{align*}
First observe that 
 \[u_\e(\bar s_j,0)=\frac 12 +\e^\frac 32 v(\bar s_j,0)=\frac 12 +\e^\frac 32 f(\bar s_j)>\frac 12\]
for $j=0,\dots,k$ so that the points $(\bar s_j,0)$ are contained in $L_{\e}$ for  every $\e$.
Next we want to prove that
\begin{equation}\label{eq:u-delta}
u_\e(s_j,y)<\frac 12
\end{equation}
for $j=1,\dots,k-1$ and $-1-h<y<1+h$.\\
In this way since  $\bar s_j<x_{2j}<s_{j+1}$ and the segment $[x_1,x_{2k}]\times \{0\}$ is contained in $\Omega_\e$ by Step 3, we also obtain that the superlevel set $L_{\e}$ admits at least $k$ disjoint components.\\
To prove \eqref{eq:u-delta} we argue by contradiction and assume that there exists 
a sequence $\e_n\to 0$ and points $y_n\in [-1-h,1+h]$ such that 
\begin{equation}\label{eq:passaggio}
u_{\e_n}(s_j,y_n)=\frac 12-\frac 12 y_n^2+\e_n(y_n^3-3s_j^2y_n)+\e_n^\frac 32 v(s_j,y_n)\ge\frac 12
\end{equation}
for $n\to \infty$, for a fixed value of $j$. Formula \eqref{eq:passaggio} easily implies that $y_n\to 0$ as $n\to \infty$ since $(y_n^3-3s_j^2y_n)$ and $v(s_j,y_n)$ are uniformly bounded and $\e_n\to 0$. Next we observe that since $v(s_j,0)=f(s_j)<0$
 then $v(s_j,y_n)\le \frac {f(s_j)} 2<0$ for $n$ large enough. Moreover,  using that $-\frac 12y_n^2+\e_n y_n^3<-\frac 14y_n^2$ for $n$ large enough we have
\[\begin{split}
u_{\e_n}(s_j,y_n)&=\frac 12-\frac 12 y_n^2+\e_n(y_n^3-3s_j^2y_n)+\e_n^\frac 32 v(s_j,y_n)\\
&\le \frac 12-\frac 14y_n^2-3\e_n s_j^2y_n+\e_n^\frac 32 \frac {f(s_j)}2 <\frac 12
\end{split}
\]
since $\e_n^\frac 32\left( \frac {f(s_j)}2+9\e_n^\frac 12s_j^4\right)<0$
 for $n$ large enough. This contradiction ends the proof.
\end{proof}

 Next aim is to derive additional information about the shape of $\Omega_\e$, in particular regarding  the oriented curvature of  $\partial \Omega_\e$. Since $\partial \Omega_\e$ is a level curve of  $u_\e(x,y)$ then its  oriented curvature at the point $(x,y)$ is given by

\begin{equation} \label{eq:curv}
Curv_{\partial \Omega_\e}(x,y)=-\frac{(u_\e)_{xx}(u_\e)_y^2-2(u_\e)_{xy}(u_\e)_x(u_\e)_y+(u_\e)_{yy}(u_\e)_x^2}{\left((u_\e)_x^2+(u_\e)_y^2\right)^\frac32}
\end{equation}
In particular, we want to prove
the following result

 \begin{lemma}\label{lem:curv-2}
The oriented curvature of  $\partial \Omega_\e$  vanishes exactly at two points when $\e$ is small enough.
\end{lemma}

Let us start examining the behavior of  some points $(\zeta_\e,\eta_\e)\in \partial \Omega_\e$ when $\e$ goes to zero.
\begin{lemma}\label{lem:behavior}
Let $(\zeta_\e,\eta_\e)$ be a point on $\partial \Omega_\e$. Then if $|\zeta_\e|\to \infty$ we have
\begin{equation}\label{e2}
|\zeta_\e|= \left(\frac 12 (1-\eta_\e^2)\right)^{\frac 1{2k}} \e^{-\frac 3{4k}}(1+o(1)).
\end{equation}
\end{lemma}
\begin{proof}
First 
$(\zeta_\e,\eta_\e)\in \partial \Omega_\e$ implies that
\begin{equation}\label{e3}
\e (\eta_\e^3-3\zeta_\e^2\eta_\e)+\e^\frac 32 v(\zeta_\e,\eta_\e)=\frac 12 (\eta_\e^2-1).
\end{equation}
Next we observe that, since $\bar \Omega_\e\subset R_\e$, where $R_\e$ is the rectangle introduced in Step 3 in the proof of Theorem \ref{b2}, then $|\zeta_\e|<3^\frac 1{2k}\e^{-\frac 3{4k}}$ and this implies that
\begin{equation}
\label{e4}
\e(\eta_\e^3-3\zeta_\e^2\eta_\e)=O\left(\e^\frac{2k-3}{2k}\right)
\end{equation}
for $\e\to 0$ and \eqref{e3} becomes, since $|\zeta_\e|\to \infty$
\[\e^\frac 32 v(\zeta_\e,\eta_\e)=\frac 12 (\eta_\e^2-1)+O\left(\e^\frac{2k-3}{2k}\right).\]
Finally by \eqref{eq:v-pj}, when $|\zeta_\e|\to \infty$ 
\[v(\zeta_\e,\eta_\e)=-\zeta_\e^{2k}(1+o(1))\]
which jointly with the previous estimate gives
\begin{equation}\label{eq:exp-x}
 \e^\frac 32 \zeta_\e^{2k}=\frac 12 (1-\eta_\e^2)(1+o(1))
 \end{equation}
when $\e\to 0$, from which \eqref{e2} follows.
\end{proof}
\vskip0.3cm
\begin{proof}[Proof of Lemma \ref{lem:curv-2}]
Here let us denote by $(\zeta_\e,\eta_\e)\in\partial\Om_\e$ a point such that $Curv_{\partial\Om_\e}(\zeta_\e,\eta_\e)=0$. By \eqref{eq:curv}
\begin{equation}\label{e9}\nonumber
Curv_{\partial\Om_\e}(\zeta_\e,\eta_\e)=-\frac {N_{1,\e}+N_{2,\e}+N_{3,\e}}{D_\e^\frac32}
\end{equation}
with
\begin{align*}
N_{1,\e}&=(-6\e \eta_\e+\e^\frac32v_{xx}(\zeta_\e,\eta_\e))
\left(-\eta_\e+3\e(\eta_\e^2-\zeta_\e^2)+\e^\frac32v_y(\zeta_\e,\eta_\e)\right)^2\\
N_{2,\e}&=-2(- 6\e \zeta_\e+\e^\frac32v_{xy}(\zeta_\e,\eta_\e))\Big(-6\e \zeta_\e \eta_\e+\e^\frac32v_x(\zeta_\e,\eta_\e)\Big)\\
&(-\eta_\e+3\e(\eta_\e^2-\zeta_\e^2)+\e^ \frac32v_y(\zeta_\e,\eta_\e))\\
N_{3,\e}&=(-1+6 \e \eta_\e+\e^\frac32v_{yy}(\zeta_\e,\eta_\e))\Big(-6\e \zeta_\e\eta_\e+\e^\frac32v_x(\zeta_\e,\eta_\e)\Big)^2\\
D_\e&= \big(-6\e\zeta_\e \eta_\e+\e^\frac32v_x(\zeta_\e,\eta_\e)\big)^2+\big(-\eta_\e+3\e(\zeta_\e^2-\eta_\e^2)+\e^\frac32v_y(\zeta_\e,\eta_\e)\big)^2.
\end{align*}
We divide the proof in some steps.
\vskip0.2cm
{\bf  Step 1: $|\zeta_\e|\to+\infty$}
\vskip0.1cm
We reason by contradiction. If the claim does not hold we can take sequences $\e_n$, $\zeta_n$, $\eta_n$ such that $\e_n\to 0$, $\zeta_n\to \zeta_0$, $\eta_n\to \eta_0$ (since $|\eta_\e|<2$ by definition of $R_\e$) and such that 
$Curv_{\partial\Om_{\e_n}}(\zeta_n,\eta_n)=0$. Since $(\zeta_n,\eta_n)\in \partial \Omega_n$ then \eqref{e3} holds and, passing to the limit, we have that  $\eta_n\to \pm1$ and 
\begin{equation}\label{eq:somma-N}
\begin{split}
&0=\frac{N_{1,\e_n}+N_{2,\e_n}+N_{3,\e_n}}\e_n=\\
&-6\big(\eta_n^3+o(1)\big)+72\e_n\big(\zeta_0^2\eta_n^2+o(1)\big)-36\e\big(\zeta_0^2\eta_n^2+o(1)\big)=-6\big(\pm1+o(1)\big)
\end{split}
\end{equation}
which gives a contradiction.
\vskip0.2cm
{\bf  Step 2}: We have that there exists $two$ values $\zeta_\e$ given by
$
\zeta_\e^\pm\sim\pm\left(\frac3{k(2k-1)\e^\frac12}\right)^\frac1{2k-2}
$
\vskip0.1cm
First, observe that, by \eqref{b1}, \eqref{eq:v-re}, \eqref{eq:v-pj}
and Step 1 we have that
\begin{align*}
&v_x=-2k\zeta_\e^{2k-1}(1+o(1)) & v_y=c_k \zeta_\e^{2k-2} \eta_\e(1+o(1))\\
& v_{xx}=-2k(2k-1) \zeta_\e^{2k-2}(1+o(1)) & v_{xy}=c'_k  \zeta_\e^{2k-3} \eta_\e(1+o(1)) \\
& v_{yy}=c_k  \zeta_\e^{2k-2}(1+o(1))
\end{align*}
where $c_k,c'_k\ne0$ are constants depending on $k$.
Using \eqref{e2} and $\e \zeta_\e^2\to 0$ as $\e \to 0$, we obtain (denoting again by $\eta_0=\lim \eta_\e$)
\[\begin{split}
&N_{1,\e}\sim \left(-6\e\eta_\e-2k(2k-1)\e^\frac 32 \zeta_\e^{2k-2}\right)\left(-\eta_\e-3\e \zeta_\e^2+c_k\e^\frac 32 \zeta_\e^{2k-2}\eta_\e
\right)^2(1+o(1))\\
&=\begin{cases}
\eta_0^2 \left( -6\e \eta_0-2k(2k-1)\e^\frac 32\zeta_\e^{2k-2}\right)(1+o(1)) & \text{ when }\eta_0\neq 0\\
o(\e+\e^\frac 32 \zeta_\e^{2k-2}) & \text{ when }\eta_0=0
\end{cases}
\end{split}\]
Using that $\e^\frac 32\zeta_\e^{2k}=O(1)$ (see \eqref{eq:exp-x}),
\[\begin{split}
&N_{2,\e}\sim -2 \left(-6\e \zeta_\e+c'_k \e^\frac 32 \zeta_\e^{2k-3}\eta_\e\right)\left(-6\e\zeta_\e\eta_\e-2k \e^\frac 32 \zeta_\e^{2k-1}\right)\cdot\\
&\left( -\eta_\e -3\e \zeta_\e^2+c_k\e^\frac 32 \zeta_\e^{2k-2}\eta_\e\right)(1+o(1))\\
&=\begin{cases}
2\eta_0^2 \left( 12k \e^{1+\frac 32} \zeta_\e^{2k}-2k c'_k\e^3 \zeta_\e^{4k-4}\eta_\e\right) (1+o(1)) & \text{ when }\eta_0\neq 0\\
o(\e+\e^3 \zeta_\e^{4k-4}) & \text{ when }\eta_0=0
\end{cases}
\end{split}\]
\[\begin{split}
&N_{3,\e}\sim \left( -1+6\e \eta_\e+c_k\e^\frac 32\zeta_\e^{2k-2}\right)\left( -6\e \zeta_\e\eta_\e-2k\e^\frac 32 \zeta_\e^{2k-1}\right)^2(1+o(1))\\
&=-1\left(4k^2\e^3\zeta_\e^{4k-2}+12 \e^{1+\frac 32}\zeta_\e^{2k}\eta_\e^2\right)(1+o(1)).
\end{split}\]
Hence if $\boxed{\lim\limits_{\e\to0}\eta_\e=\eta_0\neq \pm1}$, then $\e^\frac 32 \zeta_\e^{2k-2}\sim \e^{\frac 3{2k}}\left( \frac 12 (1-\eta_0^2)\right)^{\frac {k-1}k}$ and 
\[\begin{split}
N_{1,\e}+N_{2,\e}+N_{3,\e}
&=\e^{\frac 3{2k}}\left( -\eta_0^2 2k(2k-1) \left( \frac 12 (1-\eta_0^2)\right)^{\frac {k-1}k}\right.\\
&\left. -4k^2\left( \frac 12 (1-\eta_0^2)\right)^{\frac {2k-1}k}\right)(1+o(1))<0
\end{split}\]
showing that the curvature is {\em strictly positive} in this case.\\
So we necessarily have that  $\boxed{\lim\limits_{\e\to0} \eta_\e=\pm 1}$ and by \eqref{e2} we have that $\e^\frac 32 \zeta_\e^{2k}=o(1)$. This implies that
\[
N_{1,\e}=\left( \mp 6\e-2k(2k-1)\e^\frac 32 \zeta_\e^{2k-2}\right)(1+o(1));\]
\[
N_{2,\e}=O\left( \e^2\zeta_\e^2+\e^{1+\frac 32}\zeta_\e^{2k}+\e^3\zeta_\e^{4k-4}\right);\]
\[N_{3,\e}=O\left(  \e^3\zeta_\e^{4k-2} +\e^3\zeta_\e^{4k-4}\right)
\]
and, since $\e^2\zeta_\e^2,\e^{1+\frac 32}\zeta_\e^{2k}=o(\e)$ and $\e^3 \zeta_\e^{4k-4},\e^3\zeta_\e^{4k-2}
=o(\e^\frac 32\zeta_\e^{2k-2})$ then $N_{2,\e},N_{3,\e}
=o(N_{1,\e})$ and 
\begin{equation}\label{eq:curv-fin}
N_{1,\e}+N_{2,\e}+N_{3,\e}
=\left(\mp6\e-2k(2k-1)\e^\frac32\zeta_\e^{2k-2}\right)(1+o(1)).
\end{equation}
Since $Curv_{\partial \Omega_\e}(\zeta_\e,\eta_\e)=0$ we deduce
\begin{equation}\label{e10}
\left(\mp6\e-2k(2k-1)\e^\frac32\zeta_\e^{2k-2}\right)(1+o(1))=0
\end{equation}
as $\e\to 0$.
First we get that if $\eta_\e\to1$ \eqref{e10} does not have solutions. Hence $\eta_\e\to-1$ and \eqref{e10} becomes
\begin{equation}
\left(6-2k(2k-1)\e^\frac12\zeta_\e^{2k-2}\right)(1+o(1))=0
\end{equation}
that implies 
\begin{equation}
\e^\frac12\zeta_\e^{2k-2}=\frac3{k(2k-1)}(1+o(1)).
\end{equation}
Correspondingly we get $two$ solutions $\zeta_\e$ whose behavior is given by
\begin{equation}
\zeta_\e^\pm\sim\pm\left(\frac3{k(2k-1)\e^\frac12}\right)^\frac1{2k-2}.
\end{equation}
\vskip0.2cm
{\bf  Step 3: conclusion}
\vskip0.1cm
We end the proof showing that, corresponding to $\zeta_\e^+$ there exists only one $\eta_\e^+$ which verifies  $0=Curv_{\partial\Om_\e}(\zeta_\e,\eta_\e)=N_{1,\e}+N_{2,\e}+N_{3,\e}$ and the same is true for $\zeta_\e^-$. We apply the implicit function theorem to $u_\e$.
We have that  $u_\e(\zeta_\e^+,\eta_\e^+)=0$ and, recalling that $\eta_\e^+\to-1$,
\[
(u_\e)_y(\zeta_\e^+,\eta_\e^+)=-\eta_\e+3\e(\zeta_\e^2\e-\eta_\e^2)+\e^\frac32v_y=1+o(1)\]
for $\e\to 0$.
So by the implicit function theorem we deduce that the equation $u_\e(x,y)=0$ has only one solution for $x=\zeta^+_\e$ and $y$ close to $-1$ which ends the proof.
\end{proof}
\begin{proof}[Proof of Theorem \ref{i0}]
The existence of the family of solutions $u_{\e,k}$ to \eqref{eq:torsion} and of the domains $\Omega_{\e,k}$, as well as the properties $(P0)$ and $(P1)$ follow by Theorem \ref{b2}.\\
Property $(P2)$ is a consequence of the definition of $u_{\e,k}(x,y)$ and that locally $u_{\e,k}(x,y)\to \frac 12 (1-y^2)$ as $\e\to 0$.\\
Concerning $(P3)$, the curvature of $\partial\Omega_{\e,k}$ does change sign because, denoting by $q_\e=(0,\beta_\e)\in\partial\Omega_{\e,k}$ with $\beta_\e\to-1$, we have
$$Curv_{\partial\Om_{\e,k}}(q_\e)=\big(-6+o(1)\big)<0.$$
Next the fact that the curvature of $\partial \Omega_{\e,k}$ vanishes exactly at two points follows by Lemma \ref{lem:curv-2}. 

To prove that $\min\left(Curv_{\partial\Om_{\e,k}}\right)\to 0$ as $\e\to 0$ 
 we proceed as in the proof of Lemma \ref{lem:curv-2}. Denote by $(\tilde\zeta_\e,\tilde\eta_\e)\in\partial\Omega_\e$ a point which achieves the minimum of the curvature of $\partial\Omega_\e$ (recall that $|\tilde\eta_\e|\le C $). If, up to some subsequence, $\tilde\zeta_{\e_n}\to\tilde\zeta_0 $, since $\Omega_\e$ converges to a strip on compact set, the claim follows. On the other hand, if  $|\tilde\zeta_{\e_n}|\to+\infty$, repeating step by step the computation in Case $2$ of Lemma \ref{lem:curv-2} we again get $Curv_{\partial\Om_{\e,k}}(\tilde\zeta_{\e_n},\tilde\eta_{\e_n})\to0$. This ends the proof.
\end{proof}
\section{More general nonlinearities}
In this section we consider solutions to \eqref{f1} which satisfy \eqref{f2}. The existence is guaranteed for example if the assumptions in \cite{mp} are satisfied.  Next lemma studies the behavior as $\l\to0$.
\begin{lemma}\label{lem:convergence}
Let $u_\l$ be a family of solutions to \eqref{f1} satisfying  \eqref{f2}. Then we have that
\begin{equation}\label{f4}
\frac{u_\l}{\l f(0)}\to u_0\quad\hbox{as $\l\to0$ in }C^2(\Om)
\end{equation}
where $u_0$ is a solution to
\begin{equation}\label{f6}
\begin{cases}
-\Delta u=1&\hbox{in }\Om\\
u=0&\hbox{on }\partial\Om.
\end{cases}
\end{equation}
\end{lemma}
\begin{proof}
Let us show that 
\begin{equation}\label{f3}
|u_\l|\le C\l\quad\hbox{in }\Om
\end{equation}
where $C$ is a constant independent of $\l$. By the Green representation formula we have that
\[
|u_\l(x)|\le\l\int_\Om G(x,y)\left|f\big(u_\l(y)\big)\right|dy\le\l\max\limits_{s\in[0,C]}|f(s)|\int_\Om G(x,y)dy\le C\l
\]
where $C$ is independent of $\l$. Next by \eqref{f2}, \eqref{f3} and the standard regularity theory we derive that
\[
u_\l\to0\quad\hbox{in }C^2(\Om)
\]
as $\l\to 0$ so that $f(u_\l)\to f(0)$.
Finally the standard regularity theory, applied to $\frac{u_\l}{\l f(0)}$, and 
 \eqref{f1} gives the claim.
\end{proof}

Theorem \ref{i3} is a straightforward consequence of the previous lemma

\begin{proof}[Proof of Theorem \ref{i3}]
Assume $\e$ is small enough to satisfy the assumptions of Theorem \ref{i0}. By Lemma \ref{lem:convergence}  $\frac {u_\l}{\l f(0)}\to u_{\e,k}$  as $\l\to 0$. Then the claim follows by the $C^2$ convergence of $u_\l$ to $u_{\e,k}$ and the semi-stability of all solutions to \eqref{f6}.
\end{proof}

\bibliography{GladialiGrossiFinal.bib}
\bibliographystyle{abbrv}
\end{document}